\pdfoutput=1 

\documentclass[11pt]{amsart}

\usepackage{amssymb,amsthm,enumitem,colonequals,tikz-cd,microtype,mathtools}
\usepackage[normalem]{ulem} 

\usepackage[osf]{Baskervaldx}
\usepackage[baskervaldx]{newtxmath}
\usepackage[cal=boondoxo]{mathalfa} 

\usepackage[top=3.75cm, bottom=3cm, left=3.5cm, right=3.5cm]{geometry}

\usepackage{xcolor}
\colorlet{darkblue}{blue!55!black}
\colorlet{darkcyan}{cyan!50!black}
\colorlet{darkgreen}{green!60!black}

\PassOptionsToPackage{hyphens}{url}
\usepackage{hyperref}
\hypersetup{
    colorlinks=true,
    linkcolor=darkblue,
    urlcolor=darkcyan,
    citecolor=darkgreen,
}

\def\eqref#1{\textcolor{darkblue}{(\ref{#1})}}

\usepackage[nameinlink]{cleveref} 
\Crefformat{section}{#2\S#1#3}
\Crefmultiformat{section}{#2\S\S#1#3}{ and~#2#1#3}{, #2#1#3}{, and~#2#1#3}

\usepackage[pagewise]{lineno}
\let\oldequation\equation
\let\oldendequation\endequation
\renewenvironment{equation}{\linenomathNonumbers\oldequation}{\oldendequation\endlinenomath}
\expandafter\let\expandafter\oldequationstar\csname equation*\endcsname
\expandafter\let\expandafter\oldendequationstar\csname endequation*\endcsname
\renewenvironment{equation*}{\linenomathNonumbers\oldequationstar}{\oldendequationstar\endlinenomath}
\let\oldalign\align
\let\oldendalign\endalign
\renewenvironment{align}{\linenomathNonumbers\oldalign}{\oldendalign\endlinenomath}
\expandafter\let\expandafter\oldalignstar\csname align*\endcsname
\expandafter\let\expandafter\oldendalignstar\csname endalign*\endcsname
\renewenvironment{align*}{\linenomathNonumbers\oldalignstar}{\oldendalignstar\endlinenomath}

\theoremstyle{plain}
\newtheorem{theorem}{Theorem}[section]
\newtheorem{lemma}[theorem]{Lemma}
\newtheorem{corollary}[theorem]{Corollary}
\newtheorem{proposition}[theorem]{Proposition}

\theoremstyle{definition}
\newtheorem{definition}[theorem]{Definition}
\newtheorem{example}[theorem]{Example}
\newtheorem{remark}[theorem]{Remark}

\newtheorem{reminder}[theorem]{Reminder}

\newtheorem*{ack}{Acknowledgments}

\AddToHook{env/conjecture/begin}{\crefalias{theorem}{conjecture}}
\AddToHook{env/lemma/begin}{\crefalias{theorem}{lemma}}
\AddToHook{env/corollary/begin}{\crefalias{theorem}{corollary}}
\AddToHook{env/proposition/begin}{\crefalias{theorem}{proposition}}
\AddToHook{env/definition/begin}{\crefalias{theorem}{definition}}
\AddToHook{env/remark/begin}{\crefalias{theorem}{remark}}
\AddToHook{env/setup/begin}{\crefalias{theorem}{setup}}
\AddToHook{env/example/begin}{\crefalias{theorem}{example}}
\AddToHook{env/conjecture/begin}{\crefalias{theorem}{conjecture}}
\AddToHook{env/reminder/begin}{\crefalias{theorem}{reminder}}

\numberwithin{equation}{section}
\numberwithin{theorem}{section}

\title[Remarks on diagonal dimension for algebraic stacks]{Remarks on diagonal dimension for algebraic stacks}

\author[P.~Lank]{Pat Lank}
\address{P.~Lank,
Dipartimento di Matematica “F. Enriques”, Universit\`{a} degli Studi di Milano, Via Cesare
Saldini 50, 20133 Milano, Italy}
\email{plankmathematics@gmail.com}

\author[F. ~Peng]{Fei Peng}
\address{F.~Peng,
School of Mathematics \& Statistics\\
The University of Melbourne\\
Parkville, VIC, 3010\\
Australia}
\email{pengf2@student.unimelb.edu.au}

\date{\today}



\begin{document}
    
\begin{abstract}
    This note is concerned with the Rouquier dimension of the bounded derived category of coherent complexes on a Noetherian algebraic stack. Specifically, we study the diagonal dimension of a morphism, which can be used to produce upper bounds on Rouquier dimension. First, we obtain an explicit upper bound for smooth morphisms with a regular target. Second, we identify classical generators of a fibre product, recovering a result of Elagin--Lunts--Schn\"{u}rer. Finally, we show that the diagonal dimension of a variety in arbitrary characteristic with mild singularities is at most twice its Krull dimension.
\end{abstract}

\maketitle


\section{Introduction}
\label{sec:introduction}

\subsection{What is known}
\label{sec:introduction_what_is_known}

Let $\mathcal{X}$ be a Noetherian algebraic stack. Denote by $D^b_{\operatorname{coh}}(\mathcal{X})$ the bounded derived category of complexes with coherent cohomology. A basic problem is to understand the size and complexity of this triangulated category.

We use the notion of generation introduced by Bondal--Van den Bergh \cite{Bondal/VandenBergh:2003}. This idea leads to numerical invariants that measure the complexity \cite{Rouquier:2008, Avramov/Buchweitz/Iyengar/Miller:2010}. For $G\in D^b_{\operatorname{coh}}(\mathcal{X})$, let $\langle G \rangle_n$ denote the full subcategory consisting of objects obtained from $G$ using finite coproducts, direct summands, shifts, and at most $n-1$ cones. The \textit{Rouquier dimension} is the smallest $n \geq 0$ for which there exists $G$ such that $\langle G \rangle_{n+1} = D^b_{\operatorname{coh}}(\mathcal{X})$; such an object is called a \textit{strong generator}. An object $G$ is a \textit{classical generator} if $\cup^\infty_{n=0} \langle G \rangle_n = D^b_{\operatorname{coh}}(\mathcal{X})$.

Strong and classical generators provide a coarse description of $D^b_{\operatorname{coh}}(\mathcal{X})$. Their existence is closely related to geometric properties such as the openness of the regular locus \cite[Proposition 4.6]{DeDeyn/Lank/ManaliRahul/Peng:2025}. A central problem is to determine when such generators exist (ideally with an explicit description) and to obtain bounds on the Rouquier dimension.

For schemes, these questions are well understood. In contrast, for algebraic stacks, they present additional challenges, and effective bounds on the Rouquier dimension remain subtle. Recent work has established the existence of generators in increasing generality (see e.g.\ \cite{Hall/Priver:2024, DeDeyn/Lank/ManaliRahul:2025, DeDeyn/Lank/ManaliRahul/Peng:2025}), but explicit constructions are rare (see e.g.\ \cite{Lank/Peng:2025}), and computable bounds are even more limited.

\subsection{What we do}
\label{sec:intro_what_we_do}

Following \cite[\S 4]{Elagin/Lunts:2021}, \cite[\S 4]{DeDeyn/Lank/Rahul:2024b}, and \cite[Definition 2.15]{Ballard/Favero:2012}, we study a notion of \emph{diagonal dimension} for morphisms of algebraic stacks. It provides a refinement of the Rouquier dimension and often tractable upper bounds for the Rouquier dimension. Specifically:

\begin{definition}
    \label{def:diagonal_dimension}
    Let $f\colon \mathcal{Y}\to \mathcal{X}$ be a separated morphism of Noetherian algebraic stacks. Denote by $p_i\colon \mathcal{Y}\times_{\mathcal{X}} \mathcal{Y} \to \mathcal{Y}$ the $i$-th projection morphism and $\Delta\colon \mathcal{Y}\to \mathcal{Y}\times_{\mathcal{X}} \mathcal{Y}$ the diagonal morphism. The \textbf{diagonal dimension of $f$}, denoted $\dim_{\Delta}(f)$, is the smallest $n\geq 0$ such that $\mathbf{R} \Delta_\ast \mathcal{O}_{\mathcal{Y}}\in \langle \mathbf{L}p_1^\ast G_1 \otimes^{\mathbf{L}} \mathbf{L}p_2^\ast G_2\rangle_{n+1}\subseteq D_{\operatorname{qc}}(\mathcal{Y}\times_{\mathcal{X}} \mathcal{Y})$ for some $G_1,G_2\in D^b_{\operatorname{coh}}(\mathcal{Y})$. If no such $G_1,G_2$ exist, then we define $\dim_{\Delta}(f)=\infty$.
\end{definition} 

Note that $\mathbf{R} \Delta_\ast \mathcal{O}_{\mathcal{Y}}$ is indeed in $D_{\operatorname{coh}}^b(\mathcal{Y}\times_{\mathcal{X}} \mathcal{Y})$ because the diagonal morphism $\Delta\colon \mathcal{Y}\to \mathcal{Y}\times_{\mathcal{X}} \mathcal{Y}$ is proper by assumption. In the case of schemes over a field, the diagonal dimension provides an upper bound for the Rouquier dimension. However, it is known that the two need not coincide even for smooth curves over a field; see \cite{Olander:2024}. In fact, for stacky curves, a similar observation has been made, see \cite{Bhaduri/Goldberg/Robotis:2024}. 

\subsubsection{Bounds}
\label{sec:intro_what_we_do_bounds}

Now we state our first result. Before doing so, we need a reminder.

\begin{remark}
    Let $f\colon\mathcal{X}\to\mathcal{Y}$ be a smooth morphism of algebraic stacks. The function sending a point $x$ of $\mathcal{X}$ to $\dim_x (\mathcal{X}_{f(x)})$ is locally constant on $|\mathcal{X}|$. If $\mathcal{X}$ is quasi-compact, then this function attains a maximum value. In this case, we define $\dim (f)=\max_{x\in|\mathcal{X}|}\{\dim_x (\mathcal{X}_{f(x)})\}$. See \cite[\href{https://stacks.math.columbia.edu/tag/0DRE}{Tag 0DRQ}]{StacksProject} for details.
\end{remark}

\begin{theorem}
    \label{thm:upper_bound}
    Let $R$ be a regular ring of finite Krull dimension. Let $f\colon\mathcal{Y}\to\operatorname{Spec}(R)$ be a smooth morphism of algebraic stacks of finite presentation with finite diagonal over $R$. If $\mathcal{Y}$ has finite cohomological dimension, then we have$$\dim_{\Delta} (f)\leq\dim (f)+\operatorname{cd}(\mathcal{Y})$$and
    \begin{displaymath}
        \dim D^b_{\operatorname{coh}}(\mathcal{Y})\leq (\dim R + 1)(\dim (f)+\operatorname{cd}(\mathcal{Y}) + 1) - 1 <\infty.
    \end{displaymath}
\end{theorem}

This yields an explicit and computable upper bound for the Rouquier dimension of $D^b_{\operatorname{coh}}(\mathcal{Y})$. Here, $\operatorname{cd}(-)$ denotes the cohomological dimension of an algebraic stack (see \Cref{sec:preliminaries_algebraic_stacks}). When $R$ is a field, \Cref{thm:upper_bound} recovers known bounds for smooth varieties \cite[Proposition 7.9]{Rouquier:2008}, smooth algebraic spaces \cite[Lemma 1.2.9]{Olander:2022}, and smooth Deligne--Mumford stacks \cite[Lemma 2.20]{Ballard/Favero:2012}.

In contrast, our result applies in greater generality. For example, it applies when $\mathcal{Y}$ is a separated algebraic stack which is smooth and quasi-DM over a field of characteristic zero with affine stabilizers (see \cite[Theorem 2.1(1)]{Hall/Rydh:2015}). It also applies when $\mathcal{Y}$ is a separated quasi-DM algebraic stack which is smooth over a DVR (possibly of mixed characteristic), provided the stabilizers are affine and nice in the sense of \cite{Hall/Rydh:2015}.

\subsubsection{Fibre products}
\label{sec:intro_what_we_do_fibre_products}

Our second result identifies generators for fibre products in terms of those of the factors.

\begin{proposition}
    \label{prop:ELS_fibre_product}
    Let $\mathcal{S}$ be a concentrated quasi-DM regular algebraic stack with separated diagonal. Consider the following diagram of algebraic stacks
    \begin{displaymath}
        \begin{tikzcd}
            {\mathcal{Y}^\prime_1} & {\mathcal{Y}_1} && {\mathcal{Y}_2} & {\mathcal{Y}^\prime_2} \\
            && {\mathcal{S}}
            \arrow["{h_1}", from=1-1, to=1-2]
            \arrow[from=1-1, to=2-3]
            \arrow["{f_1}", from=1-2, to=2-3]
            \arrow["{f_2}"', from=1-4, to=2-3]
            \arrow["{h_2}"', from=1-5, to=1-4]
            \arrow[from=1-5, to=2-3]
        \end{tikzcd}
    \end{displaymath}
    such that
    \begin{enumerate}
        \item the $f_i$ are concentrated, flat, finitely presented, and surjective,
        \item the $h_i$ are proper, finitely presented, and surjective, and
        \item the $f_i\circ h_i$ are smooth, concentrated, and quasi-DM.
    \end{enumerate}
    If $G_i$ is a classical generator for $D^b_{\operatorname{coh}}(\mathcal{Y}_i)$ for each $i$, then $\mathbf{L}g_1^\ast G_1 \otimes^{\mathbf{L}} \mathbf{L}g_2^\ast G_2$ is a classical generator for $D^b_{\operatorname{coh}}(\mathcal{Y}_1\times_{\mathcal{S}} \mathcal{Y}_2)$ where $g_i \colon \mathcal{Y}_1\times_{\mathcal{S}} \mathcal{Y}_2 \to \mathcal{Y}_i$ are the natural projections.
\end{proposition}

Notably, \Cref{prop:ELS_fibre_product} applies broadly and does not require the base $\mathcal{S}$ to be separated or of finite Krull dimension. For instance, it applies to separated integral schemes of finite type over a perfect field via \cite[Theorem 4.1]{deJong:1996}, and to reduced Deligne--Mumford stacks of finite presentation over a field of characteristic zero via \cite[\S 5.1]{Temkin:2012} and \cite[\S 3.6]{Edidin/Rydh:2021}. The separated diagonal assumption, however, is necessary for us to ensure that $\mathcal{S}$ has quasi-affine diagonal.

Moreover, our approach recovers \cite[Theorem 4.18]{Elagin/Lunts/Schnurer:2020} for varieties over a perfect field and provides a proof without dg methods. Unlike loc.\ cit., we work directly at the level of algebraic stacks and extend to Deligne--Mumford stacks admitting suitable regular covers. It would be interesting for future work to study analogs of our results for coherent noncommutative algebras over an algebraic stack (see e.g.\ \cite[Theorem 4.15]{Elagin/Lunts/Schnurer:2020}).

\subsubsection{Comparisons}
\label{sec:intro_what_we_do_comparions}

Our final result compares the diagonal dimension along morphisms whose unit for derived pullback/pushforward is an isomorphism.

\begin{theorem}
    \label{thm:compare_diagonal_morphism}
    Let $\mathcal{S}$ be a Noetherian algebraic stack. Consider concentrated finitely presented flat and separated morphisms $f_i\colon \mathcal{Y}_i \to \mathcal{S}$. If there exists a concentrated $\mathcal{S}$-morphism $f\colon \mathcal{Y}_2 \to \mathcal{Y}_1$ such that $\mathcal{O}_{\mathcal{Y}_1}\xrightarrow{ntrl.} \mathbf{R} f_\ast \mathcal{O}_{Y_2}$ is an isomorphism and $\mathbf{R}f_\ast$ preserves coherent cohomology (e.g. $f$ is proper), then $\dim_{\Delta}(f_1) \leq \dim_{\Delta}(f_2)$.
\end{theorem}

There are important situations in which \Cref{thm:compare_diagonal_morphism} applies. 

\begin{corollary}
    \label{cor:compare_diagonal_morphism}
    \hfill
    \begin{enumerate}
        \item (characteristic zero) Let $f\colon X \to \operatorname{Spec}(k)$ be a morphism of finite type from a separated integral scheme to a field of characteristic zero. If $X$ has rational singularities, then $\dim_{\Delta}(f)\leq 2 \dim X$. 
        \item (arbitrary characteristic) Let $f\colon X\to \operatorname{Spec}(k)$ be a separated irreducible scheme of finite type over a perfect field $k$. If $X$ has at worst linearly reductive singularities, then $\dim_{\Delta}(f)\leq 2 \dim X$.
    \end{enumerate}
\end{corollary}

These results appear to provide a new general bound on the diagonal dimension for varieties with mild singularities. In the case of characteristic zero, \cite{Hanlon/Hicks/Lazarev:2024} and \cite{Favero/Huang:2023} have shown that the diagonal dimension for many toric varieties is precisely the Krull dimension. However, \Cref{cor:compare_diagonal_morphism} holds for cases which need not be toric and provides an upper bound (e.g.\ quotient singularities, see \Cref{ex:char_zero_non_toric}). Moreover, in positive characteristics, our results are applicable to varieties with tame quotient singularities. See \Cref{rmk:quotient_singularities}.

\begin{ack}
    The authors thank Anirban Bhaduri, Timothy De Deyn, Tianqi Feng, Jack Hall, and Oliver Li for discussions. Also, we greatly appreciate discussions with Alexey Elagin, which motivated \Cref{thm:compare_diagonal_morphism}. Lank was supported by the ERC Advanced Grant 101095900-TriCatApp. Peng was supported by the Australian Research Council DP210103397 and FT210100405, a Melbourne Research Scholarship, and the Science Abroad Traveling Scholarship offered by the University of Melbourne. We appreciate the referee for feedback given on an earlier version.
\end{ack}

\section{Preliminaries}
\label{sec:preliminaries}

\subsection{Generation}
\label{sec:preliminaries_generation}

We discuss generation and dimension for triangulated categories. See \cite{Bondal/VandenBergh:2003, Rouquier:2008} for details. Let $\mathcal{T}$ be a triangulated category with shift functor $[1]\colon \mathcal{T} \to \mathcal{T}$. Consider a subcategory $\mathcal{S} \subseteq \mathcal{T}$. A triangulated subcategory of $\mathcal{T}$ is called \textbf{thick} if it is closed under direct summands. Denote by $\langle \mathcal{S} \rangle$ the smallest thick subcategory of $\mathcal{T}$ containing $\mathcal{S}$; if $\mathcal{S}$ consists of a single object $G$, we write $\langle \mathcal{S} \rangle := \langle G \rangle$. Set $\operatorname{add}(\mathcal{S})$ to be the smallest strictly full subcategory of $\mathcal{T}$ containing $\mathcal{S}$ that is closed under shifts, finite coproducts, and direct summands. Inductively, let $\langle \mathcal{S} \rangle_0$ consist of all objects in $\mathcal{T}$ isomorphic to the zero object, $\langle \mathcal{S} \rangle_1 := \operatorname{add}(\mathcal{S})$, and 
\begin{displaymath}
    \langle \mathcal{S} \rangle_n := \operatorname{add} \{ \operatorname{cone}(\phi) \mid \phi \in \operatorname{Hom}_{\mathcal{T}} (\langle \mathcal{S} \rangle_{n-1}, \langle \mathcal{S} \rangle_1) \}.
\end{displaymath}
It can be checked that $\langle \mathcal{S} \rangle = \cup_{n=0}^\infty \langle \mathcal{S} \rangle_n$. We call an object $G \in \mathcal{T}$ a \textbf{classical generator} if $\langle G \rangle = \mathcal{T}$. Moreover, if there exists $n \ge 0$ such that $\langle G \rangle_{n+1} = \mathcal{T}$, then $G$ is called a \textbf{strong generator}. The smallest integer $n \ge 0$ satisfying $\mathcal{T} = \langle G \rangle_{n+1}$ for some $G \in \mathcal{T}$ is called the \textbf{Rouquier dimension} of $\mathcal{T}$ and is denoted by $\dim \mathcal{T}$.

Now, assume that $\mathcal{T}$ admits all small coproducts. Let $\operatorname{Add}(\mathcal{S})$ be the smallest strictly full subcategory of $\mathcal{T}$ containing $\mathcal{S}$ that is closed under shifts, small coproducts, and direct summands. Inductively, let $\overline{\langle \mathcal{S} \rangle}_0$ consist of all objects in $\mathcal{T}$ isomorphic to the zero object, $\overline{\langle \mathcal{S} \rangle}_1 := \operatorname{Add}(\mathcal{S})$, and
\begin{displaymath}
    \overline{\langle \mathcal{S} \rangle}_n := \operatorname{Add} \{ \operatorname{cone}(\phi) \mid \phi \in \operatorname{Hom}_{\mathcal{T}} (\overline{\langle \mathcal{S} \rangle}_{n-1}, \overline{\langle \mathcal{S} \rangle}_1) \}.
\end{displaymath}

\begin{example}
    \label{ex:regular_ring}
    Let $X=\operatorname{Spec}(R)$ where $R$ is a regular ring of finite Krull dimension. Then $\overline{\langle\mathcal{O}_X \rangle}_{\dim X + 1} = D_{\operatorname{qc}}(X)$. On some level, this goes as far back as \cite{Kelly:1965}, \cite{Street:1973}, and \cite[Corollary 8.4]{Christensen:1998}; see \cite[Corollary 4.3.13]{Letz:2020} for a modern treatment.
\end{example}

The collection of compact objects in $\mathcal{T}$ will be denoted by $\mathcal{T}^c$. These form a triangulated subcategory of $\mathcal{T}$. When $\mathcal{T}$ admits small coproducts, we say $\mathcal{T}$ is \textbf{compactly generated} if it coincides with the smallest triangulated subcategory of $\mathcal{T}$ containing $\mathcal{T}^c$ and closed under small coproducts. This is equivalent to the condition that, for any $E \in \mathcal{T}$ satisfying $\operatorname{Hom}(P, E) = 0$ for all $P \in \mathcal{T}^c$, one has $E \cong 0$ \cite[Lemma 2.2.1]{Schwede/Shipley:2003}. Note that classical generators for $\mathcal{T}^c$ coincide with compact generators for $\mathcal{T}$ \cite[\href{https://stacks.math.columbia.edu/tag/09SR}{Tag 09SR}]{StacksProject}.

\begin{example}
    Let $X$ be a quasi-compact quasi-separated scheme. Then $D_{\operatorname{qc}}(X)^c=\operatorname{Perf}(X)$, and moreover, $\operatorname{Perf}(X)$ admits a classical generator. See \cite[Theorem 3.1.1]{Bondal/VandenBergh:2003}.
\end{example}

\subsection{Algebraic stacks}
\label{sec:preliminaries_algebraic_stacks}

We follow \cite{StacksProject} for conventions on algebraic stacks. For the derived pullback/pushforward adjunction, we follow \cite[\S1]{Hall/Rydh:2017} and \cite{Olsson:2007a,Laszlo/Olsson:2008a,Laszlo/Olsson:2008b}. Unless otherwise specified, symbols such as $X$, $Y$, etc.\ denote schemes or algebraic spaces, while $\mathcal{X}$, $\mathcal{Y}$, etc.\ denote algebraic stacks. In this subsection, let $\mathcal{X}$ be a quasi-compact quasi-separated algebraic stack.

\subsubsection{Categories}
\label{sec:prelim_stacks_categories}

$\operatorname{Mod}(\mathcal{X})$ is the Grothendieck abelian category of sheaves of $\mathcal{O}_\mathcal{X}$-modules on the lisse-\'{e}tale site of $\mathcal{X}$. $\operatorname{Qcoh}(\mathcal{X})$ is the full subcategory of $\operatorname{Mod}(\mathcal{X})$ consisting of quasi-coherent sheaves. $D(\mathcal{X}) := D(\operatorname{Mod}(\mathcal{X}))$ is the derived category of $\operatorname{Mod}(\mathcal{X})$. By $D_{\operatorname{qc}}(\mathcal{X})$ (resp. $D_{\operatorname{coh}}^b(\mathcal{X})$), we mean the full subcategory of $D(\mathcal{X})$ consisting of complexes (resp. bounded complexes) with quasi-coherent (resp. coherent) cohomology sheaves. By $\operatorname{Perf}(\mathcal{X})$, we mean the full subcategory of perfect complexes in $D_{\operatorname{qc}}(\mathcal{X})$. If $\mathcal{X}$ is Noetherian and affine-pointed, then $\operatorname{coh}(\mathcal{X})$ is the full subcategory of $\operatorname{Mod}(\mathcal{X})$ consisting of coherent sheaves and there is an equivalence $D^b_{\operatorname{coh}}(\mathcal{X})\cong D^b(\operatorname{coh}(\mathcal{X}))$. Indeed, we have an equivalence $D_{\operatorname{Coh}(\mathcal{X})}^{b}(\operatorname{QCoh}(\mathcal{X}))\cong D^{b}_{\operatorname{Coh}}(\mathcal{X})$ by restricting the equivalence in \cite[Theorem C.1]{Hall/Neeman/Rydh:2019}. The equivalence $D^{b}(\operatorname{Coh}(\mathcal{X}))\cong D_{\operatorname{Coh}(\mathcal{X})}^{b}(\operatorname{QCoh}(\mathcal{X}))$ follows from a standard argument (e.g. \cite[Proposition 3.5]{Huybrechts:2006}). Denote by $\operatorname{\mathbb{R}\mathcal{H}\! \mathit{om}}$ (resp.\ $\operatorname{\mathbf{R}\mathcal{H}\! \mathit{om}}$) the `internal hom' for $D(\mathcal{X})$ (resp.\ $D_{\operatorname{qc}}(\mathcal{X})$).

\subsubsection{Concentratedness}
\label{sec:prelim_stacks_concentrated}

A morphism of algebraic stacks is called \textbf{concentrated} if it is quasi-compact, quasi-separated, and if the derived pushforward of any base change along a quasi-compact quasi-separated morphism has finite cohomological dimension. An algebraic stack is \textbf{concentrated} if it is quasi-compact and quasi-separated, and its structure morphism to $\operatorname{Spec}(\mathbb{Z})$ is concentrated. In fact, a quasi-compact quasi-separated algebraic stack $\mathcal{X}$ is concentrated if, and only if, any (hence all) of the following equivalent conditions hold: $\operatorname{Perf}(\mathcal{X}) = D_{\operatorname{qc}}(\mathcal{X})^c$, $\mathcal{O}_{\mathcal{X}}\in D_{\operatorname{qc}}(\mathcal{X})^c$, or $\operatorname{cd}(\mathcal{X})<\infty$. Here, the \textbf{cohomological dimension} of $\mathcal{X}$ is defined as
\begin{displaymath}
    \operatorname{cd}(\mathcal{X}) := \inf\left\{ d \mid \forall E\in \operatorname{Qcoh}(\mathcal{X})\ \text{and}\ n>d,\; H^n(\mathcal{X}, E) =0\right\}.
\end{displaymath}
See \cite[\S2, Proposition 4.5 and Remark 4.6]{Hall/Rydh:2017} for details.

\subsubsection{Perfect complexes}
\label{sec:prelim_stacks_perfects}

Perfect complexes are defined on any ringed site \cite[\href{https://stacks.math.columbia.edu/tag/08G4}{Tag 08G4}]{StacksProject}, for example on the lisse-\'{e}tale site of $\mathcal{X}$. A complex is \textbf{strictly perfect} if it is a bounded complex whose terms are direct summands of finite free modules; it is \textbf{perfect} if it is locally strictly perfect. Let $\operatorname{Perf}(\mathcal{X})$ denote the triangulated subcategory of $D_{\operatorname{qc}}(\mathcal{X})$ consisting of perfect complexes. In general, the compact objects of $D_{\operatorname{qc}}(\mathcal{X})$ are perfect complexes if $\mathcal{X}$ is quasi-compact and quasi-separated \cite[Lemma 4.4]{Hall/Rydh:2017}, although the converse need not hold. The two notions coincide precisely when the algebraic stack $\mathcal{X}$ is concentrated \cite[Lemma 4.4]{Hall/Rydh:2017}.

\section{Proof of results}
\label{sec:proof_of_results}

In this section, we prove our main results. To prove \Cref{thm:upper_bound}, we need a few lemmas. 

\begin{lemma}
    \label{lem:alt_Delta_dim}
    With notation of \Cref{def:diagonal_dimension}, if $\dim_\Delta (f)<\infty$, then it is equal to the smallest integer $n$ for which $\mathbf{R}\Delta_\ast \mathcal{O}_{\mathcal{Y}}\in\langle \mathbf{L}p_1^\ast G \otimes^{\mathbf{L}} \mathbf{L}p_2^\ast G \rangle_{n+1}$ for $G\in D^b_{\operatorname{coh}}(\mathcal{Y})$.
\end{lemma}

\begin{proof}
    Set $m:=\dim_\Delta (f)$. Clearly, $m\leq n$. Hence, it suffices to show the reverse inequality. Suppose that $\mathbf{R}\Delta_\ast \mathcal{O}_{\mathcal{Y}}\in\langle\mathbf{L}p_1^\ast G_1 \otimes^{\mathbf{L}} \mathbf{L}p_2^\ast G_2 \rangle_{m+1}$ for some $G_1,\ G_2\in D^b_{\operatorname{coh}}(\mathcal{Y})$. Then $\mathbf{L}p_1^\ast G_1 \otimes^{\mathbf{L}} \mathbf{L}p_2^\ast G_2\in \langle \mathbf{L}p_1^\ast G \otimes^{\mathbf{L}} \mathbf{L}p_2^\ast G \rangle_1$ where $G:=G_1 \oplus G_2$. Consequently, $\mathbf{R}\Delta_\ast \mathcal{O}_{\mathcal{Y}}\in\langle \mathbf{L}p_1^\ast G_1 \otimes^{\mathbf{L}} \mathbf{L}p_2^\ast G_2 \rangle_{m+1}$ implies $\mathbf{R}\Delta_\ast \mathcal{O}_{\mathcal{Y}}\in\langle \mathbf{L}p_1^\ast G \otimes^{\mathbf{L}} \mathbf{L}p_2^\ast G \rangle_{m+1}$, and hence, $n\leq m$.
\end{proof}

\begin{remark}
    With \Cref{lem:alt_Delta_dim} in mind, we can now replace the $G_i$ in \Cref{def:diagonal_dimension} by a single object. 
\end{remark}

\begin{lemma}
    \label{lem:spectral_sequence}
    Let $\mathcal{X}$ be a quasi-compact quasi-separated algebraic stack. If $L\in D^+_{\operatorname{qc}}(\mathcal{X})$ and $K$ is pseudocoherent, then there exists a spectral sequence
    \begin{displaymath}
        E^{p,q}_2 := H^p (\mathcal{X}, \mathcal{H}^q (\operatorname{\mathbf{R}\mathcal{H}\! \mathit{om}} (K,L)))
    \end{displaymath}
    which converges to $\operatorname{Ext}^{p+q}(K,L)$.
\end{lemma}

\begin{proof}
    Starting from the left exact functor
    \begin{displaymath}
        \Gamma(\mathcal{X},-)\colon \operatorname{Mod}(\mathcal{X})\to \operatorname{Mod}(\Gamma(\mathcal{X},\mathcal{O}_{\mathcal{X}})),
    \end{displaymath}
    we obtain 
    \begin{displaymath}
        \mathbf{R}\Gamma(\mathcal{X},-)\colon D (\mathcal{X}) \to D (\Gamma(\mathcal{X},\mathcal{O}_{\mathcal{X}})).
    \end{displaymath}
    See \cite[\href{https://stacks.math.columbia.edu/tag/071J}{Tag 071J} \& \href{https://stacks.math.columbia.edu/tag/07A5}{07A5}]{StacksProject}. Note that $H^i (\mathcal{X},-) := \mathbf{R}^i \Gamma(\mathcal{X},-)$. Since $L\in D^+_{\operatorname{qc}}(\mathcal{X})$ and $K$ is pseudocoherent, it follows that $\operatorname{\mathbb{R}\mathcal{H}\! \mathit{om}}(K,L)\in D^+_{\operatorname{qc}}(\mathcal{X})$. Indeed, this can be checked smooth locally with \cite[\href{https://stacks.math.columbia.edu/tag/0A6H}{Tag 0A6H}]{StacksProject}. Recall that 
    \begin{displaymath}
        \operatorname{Ext}^p (K,L) = H^p (\mathbf{R}\Gamma(\mathcal{X},\mathbb{R}\mathcal{H}\! \mathit{om} (K,L))).
    \end{displaymath}
    See \cite[\href{https://stacks.math.columbia.edu/tag/0B6E}{Tag 0B6E}]{StacksProject}. As $\operatorname{Mod}(\mathcal{X})$ is a Grothendieck abelian category with enough injectives, $\operatorname{\mathbb{R}\mathcal{H}\! \mathit{om}}(K,L)$ admits a Cartan--Eilenberg resolution (see \cite[\href{https://stacks.math.columbia.edu/tag/07A5}{Tags 07A5} and \href{https://stacks.math.columbia.edu/tag/015I}{015I}]{StacksProject}). Therefore, by \cite[\href{https://stacks.math.columbia.edu/tag/015J}{Tag 015J}]{StacksProject}, the desired spectral sequence exists.
\end{proof}

\begin{lemma}
    \label{lem:exi^q_vanishing_dim_U}
    Let $R$ be a regular ring of finite Krull dimension, $\mathcal{X}$ be a smooth algebraic stack with finite diagonal over $R$, and $\Delta\colon\mathcal{X}\to\mathcal{X}\times_R\mathcal{X}$ be the diagonal morphism. If $E$ is a coherent sheaf on $\mathcal{X}\times_R\mathcal{X}$, then $\operatorname{\mathcal{E}\! \mathit{xt}}^q(\Delta_\ast \mathcal{O_X},E)=0$ for all $q>\dim(g)$ where $g\colon \mathcal{X}\to \operatorname{Spec}(R)$ is the structure morphism.
\end{lemma}

\begin{proof}
    Choose a smooth presentation $s\colon U\to\mathcal{X}$. Consider the following fibred square
    \begin{displaymath}
        \begin{tikzcd}
            {W} & {V} \\
            {\mathcal{X}} & {\mathcal{X}\times_{\operatorname{Spec}(R)} \mathcal{X}.}
            \arrow["j", from=1-1, to=1-2]
            \arrow[from=1-1, to=2-1]
            \arrow["s\times_R s", from=1-2, to=2-2]
            \arrow["\Delta"', from=2-1, to=2-2]
        \end{tikzcd}
    \end{displaymath}
    where $V=U\times_R U$ and $W=U\times_\mathcal{X}U$. Note that $j$ is a finite morphism between regular schemes. Indeed, the composition $U\to\mathcal{X}\to\operatorname{Spec}(R)$ is smooth and $R$ is regular. It follows that $U$ and $W$ are both regular. Moreover, the fibre product $\mathcal{X}\times_R\mathcal{X}$ is also regular. Since $s\times_R s$ is smooth, we see that $V$ is also regular.
    
    We claim that $\operatorname{\mathcal{E}\! \mathit{xt}}^q(j_\ast \mathcal{O}_W, G)=0$ if $q>\dim(g)$. We first show that the Lemma follows from our claim. Note that $\Delta_\ast \mathcal{O_X}$ is perfect because it is coherent and $\mathcal{X}\times_R\mathcal{X}$ is regular. By flat base change and our claim, we have
    \begin{align*}
        (s\times_Rs)^\ast \operatorname{\mathcal{E}\! \mathit{xt}}^q(\Delta_\ast \mathcal{O_X},E)&\simeq \mathbf{L}(s\times_Rs)^\ast \operatorname{\mathcal{E}\! \mathit{xt}}^q(\Delta_\ast \mathcal{O_X},E)\\
        &\simeq\operatorname{\mathcal{E}\! \mathit{xt}}^q(\mathbf{L}(s\times_Rs)^\ast \Delta_\ast \mathcal{O_X},\mathbf{L}(s\times_Rs)^\ast E)\\
        &\simeq\operatorname{\mathcal{E}\! \mathit{xt}}^q(j_\ast \mathcal{O}_W,\mathbf{L}(s\times_Rs)^\ast E)\\
        &=0
    \end{align*}
    Since $s\times_Rs$ is faithfully flat, we see that $\operatorname{\mathcal{E}\! \mathit{xt}}^q(\Delta_\ast \mathcal{O_X},E)=0$ as desired.
    
    We now prove our claim. Let $G$ be a coherent sheaf on $V$. Let $v$ be a closed point of $V$. Consider $\operatorname{\mathcal{E}\! \mathit{xt}}^q(j_\ast \mathcal{O}_W,G)_v\cong\operatorname{\mathcal{E}\! \mathit{xt}}^q((j_\ast \mathcal{O}_W)_v,G_v)$. Note that $(j_\ast \mathcal{O}_W)_v$ is a finite $\mathcal{O}_{V,v}$-module. Let $A=\mathcal{O}_{V,v}$ and $B=(j_\ast \mathcal{O}_W)_v$. If $v$ is not in the image of $j$, then $B=0$ and there is nothing to prove. Otherwise, there are finitely many closed points $w_i,\ldots,w_n$ of $W$ lying over $v$. By assumption, $A$ is a regular local ring and $B$ is a finite semilocal $A$-algebra. In particular, $B$ has finite projective dimension as an $A$-module. By Auslander-Buchsbaum formula \cite[\href{https://stacks.math.columbia.edu/tag/090V}{Tag 090V}]{StacksProject}, we have $\operatorname{pd}_{A}(B)=\operatorname{depth}(A)-\operatorname{depth}_A(B)$. Since $V$ is regular, we see that $\operatorname{depth}(A)=\dim(A)$. For $\operatorname{depth}_AB$, let $B_i=\mathcal{O}_{W,w_i}$ for every $1\leq i\leq n$. Since $W$ is regular, we see that each $B_i$ is a regular local ring and $\operatorname{depth}(B_i)=\dim(B_i)$. By \cite[\href{https://stacks.math.columbia.edu/tag/0AUK}{Tag 0AUK}]{StacksProject}, we have $\operatorname{depth}_AB=\min_{1\leq i\leq n}\dim(B_i)$. It follows that $\operatorname{pd}_A(B)=\max_{1\leq i\leq n}\{\dim(A)-\dim(B_i)\}$. Let $u=t(v)=r(w_{i})$ in $U$ where $t\colon V\to U$ and $r\colon W\to U$ are the canonical projections. Note that both $t$ and $r$ are flat. This implies that $\dim(A)=\dim (\mathcal{O}_{U,u})+\dim(\mathcal{O}_{V_u,v})$ where $V_u$ is the fibre of $u$ in $V$. Similarly, we have $\dim(B_i)=\dim (\mathcal{O}_{U,u})+\dim(\mathcal{O}_{W_u,w_i})$ where $W_u$ is the fibre of $u$ in $W$. Moreover, let $x=s(u)$. It follows that
    \begin{align*}
        \operatorname{pd}_A(B)&=\max_{1\leq i\leq n}\{\dim(A)-\dim(B_i)\}\\
        &=\max_{1\leq i\leq n}\{\dim(\mathcal{O}_{V_u,v})-\dim(\mathcal{O}_{W_u,w_i})\}\\
        &=\dim_x(\mathcal{X}_{g(x)})\\
        &\leq\dim(g).
    \end{align*}
    Note that the second equality follows from applying \cite[\href{https://stacks.math.columbia.edu/tag/0DRP}{Tag 0DRP}]{StacksProject} to the smooth morphism $V_u\to\mathcal{X}_{\kappa(u)}$ whose fibre at $\operatorname{Spec}(\kappa(u))\to\mathcal{X}_x$ is precisely $W_u$. This implies that $\operatorname{\mathcal{E}\! \mathit{xt}}^q(j_\ast \mathcal{O}_W,G)_v=0$ if $q>\dim(g)$. Since this is true for every closed point $v$ of $V$, we have $\operatorname{\mathcal{E}\! \mathit{xt}}^q(j_\ast \mathcal{O}_W,G)=0$ if $q>\dim(g)$. This finishes the proof.
\end{proof}

\begin{lemma}\label{lem:exi^D_vanishing_regular_finite_Krull_dimension}
    Let $R$ be a regular ring of finite Krull dimension, $g\colon \mathcal{X} \to \operatorname{Spec}(R)$ be a smooth morphism from a concentrated algebraic stack with finite diagonal, and $\Delta\colon\mathcal{X}\to\mathcal{X}\times_R\mathcal{X}$ be the diagonal morphism. Set $D=\dim(g)+\operatorname{cd}(\mathcal{X})$. If $E$ is coherent sheaf on $\mathcal{X}\times_R\mathcal{X}$, then $\operatorname{Ext}^n(\Delta_\ast \mathcal{O_X},E)=0$ for all $n>D$. Moreover, if $K\in D^b_{\operatorname{coh}}(\mathcal{X}\times_R\mathcal{X})^{\leq 0}$, then $\operatorname{Hom}(\mathbf{R}\Delta_\ast \mathcal{O_X},K[n])=0$ for every $n>D$.
\end{lemma}

\begin{proof}
    Note that $D<\infty$ by assumption. By \Cref{lem:spectral_sequence}, we have the following local-to-global spectral sequence
    \begin{displaymath}
        E^{p,q} := H^p (\mathcal{X}\times_R\mathcal{X}, \operatorname{\mathcal{E}\! \mathit{xt}}^q (\Delta_\ast \mathcal{O_X},E))\implies\operatorname{Ext}^{p+q}(\Delta_\ast \mathcal{O_X},E).
    \end{displaymath}
    By \Cref{lem:exi^q_vanishing_dim_U}, we have $\operatorname{\mathcal{E}\! \mathit{xt}}^q (\Delta_\ast \mathcal{O_X},E)=0$ if $q>\dim(g)$. Moreover, we have $H^p (\mathcal{X}\times_R\mathcal{X}, \operatorname{\mathcal{E}\! \mathit{xt}}^q (\Delta_\ast \mathcal{O_X},E))=0$ if $p>\operatorname{cd}(\mathcal{X})$. Indeed, $\operatorname{\mathcal{E}\! \mathit{xt}}^q (\Delta_\ast \mathcal{O_X},E)$ has a natural $\Delta_\ast \mathcal{O_X}$-module structure and it is coherent because $\Delta_\ast \mathcal{O_X}$ is perfect and $E$ is coherent. Since $\Delta$ is affine, there exists a coherent sheaf $F_q$ on $\mathcal{X}$ such that $\operatorname{\mathcal{E}\! \mathit{xt}}^q (\Delta_\ast \mathcal{O_X},E)\cong\Delta_\ast F_q$. Therefore, we have
    \begin{displaymath}
        H^p (\mathcal{X}\times_R\mathcal{X}, \operatorname{\mathcal{E}\! \mathit{xt}}^q (\Delta_\ast \mathcal{O_X},E))\cong H^p (\mathcal{X}\times_R\mathcal{X}, \Delta_\ast F_q)\cong H^p (\mathcal{X}, F_q).
    \end{displaymath}
    Therefore, we see that these cohomology groups vanish if $p>\operatorname{cd}(\mathcal{X})$ as desired. It follows that $\operatorname{Ext}^{n}(\Delta_\ast \mathcal{O_X},E)=0$ if $n>\operatorname{cd}(\mathcal{X})+\dim(g)$ as desired. 

    Let $K\in D^b_{\operatorname{coh}}(\mathcal{X}\times_R\mathcal{X})^{\leq 0}$. Consider the following hypercohomology spectral sequence
    \begin{displaymath}
        E_2^{p,q}=\operatorname{Ext}^p(\mathbf{R}\Delta_\ast \mathcal{O_X},\mathcal{H}\!^q(K))\implies\operatorname{Hom}(\mathbf{R}\Delta_\ast \mathcal{O_X},K[p+q]).
    \end{displaymath}
    Suppose $n=p+q>D$. Since $K\in D^b_{\operatorname{coh}}(\mathcal{X}\times_R\mathcal{X})^{\leq 0}$, we see that $\mathcal{H}\!^q(K)$ is nonzero only when $q\leq 0$. Therefore we have $p=n-q>D$. But in this case, we have $\operatorname{Ext}^p(\mathbf{R}\Delta_\ast \mathcal{O_X},E)=0$ for every coherent sheaf $E$ on $\mathcal{X}\times_R\mathcal{X}$. The spectral sequence then tells us that $\operatorname{Hom}(\mathbf{R}\Delta_\ast \mathcal{O_X},K[n])=0$ as desired.
\end{proof}

\begin{example}
    \label{ex:cohomological_dimension_good_moduli_space}
    Let $\mathcal{X}$ be a quasi-compact and quasi-separated algebraic stack. Let $\pi \colon \mathcal{X}\to X$ be a good moduli space. If $\dim X \leq d$, then $H^n (\mathcal{X},E)\cong 0$ for all $n > \dim X$ and $E\in \operatorname{Qcoh}(\mathcal{X})$. In other words, $\operatorname{cd}(\mathcal{X})\leq d$. Indeed, there is an isomorphism $H^n (\mathcal{X},E)\cong H^n (X,\pi_\ast E)$ (see e.g.\ \cite[Eq.\ 1.10]{Hall/Rydh:2017}). By \cite[\href{https://stacks.math.columbia.edu/tag/0A4R}{Tag 0A4R}]{StacksProject}, $H^n (X,\pi_\ast E) = 0$ for $n > \dim X$. 
\end{example}

\begin{lemma}
    \label{lem:big_to_small_generation}
    Let $\mathcal{X}$ be a concentrated Noetherian algebraic stack with quasi-finite and separated diagonal.
    If $E\in\overline{\langle G \rangle}_n$ for some $E,G\in D^b_{\operatorname{coh}}(\mathcal{X})$ and $n\geq 0$, then $E\in \langle G \rangle_n$.
\end{lemma}

\begin{proof}
    This follows from a string of results and concepts in the literature (see e.g.\ \cite{Neeman:2021} and the notation in loc.\ cit.). By \cite[Theorem A]{Hall/Rydh:2017}, $D_{\operatorname{qc}}(\mathcal{X})$ is singly compactly generated. Moreover, \cite[Proposition 6.6]{DeDeyn/Lank/ManaliRahul/Peng:2025} shows that the standard $t$-structure on $D_{\operatorname{qc}}(\mathcal{X})$ is in the preferred equivalence class. Since $\mathcal{X}$ is concentrated, we have $\operatorname{Perf}(\mathcal{X})=D_{\operatorname{qc}}(\mathcal{X})^c\subseteq \mathcal{T}^b_c$ where $\mathcal{T}:= D_{\operatorname{qc}}(\mathcal{X})$ and $\mathcal{T}^b_c$ is defined as in \cite{Neeman:2021}. Furthermore, \cite[Theorem A]{Hall/Lamarche/Lank/Peng:2025} implies that $D^b_{\operatorname{coh}}(\mathcal{X})\subseteq \mathcal{T}^b_c$. The claim now follows from \cite[Lemma 2.14]{DeDeyn/Lank/ManaliRahul:2024}.
\end{proof}

\begin{lemma}
    \label{lem:derived_resolution_of_the_diagonal}
    Let $R$ be a regular ring of finite Krull dimension and $\mathcal{X}$ be a Noetherian algebraic stack with quasi-finite and separated diagonal over $R$. Let $K$ be a complex in $D_{\operatorname{coh}}^b(\mathcal{X}\times_R\mathcal{X})$. Let $g_1,g_2\colon \mathcal{X}\times_R\mathcal{X}\to\mathcal{X}$ be the canonical projections. If $\mathcal{X}$ is smooth and finitely presented over $R$, then there exists a distinguished triangle in $D_{\operatorname{coh}}^b(\mathcal{X}\times_R\mathcal{X})$,
    \begin{displaymath}
        B_1\longrightarrow K_0\longrightarrow K\longrightarrow B_1[1]
    \end{displaymath}
    such that $K_0\to K$ induces surjective maps on the cohomology sheaves and
    \begin{displaymath}
        K_0=\bigoplus_{j\in J}\mathbf{L}g_1^\ast F_j\otimes^{\mathbf{L}}\mathbf{L}g_2^\ast G_j[n_j]
    \end{displaymath}
    for some finite set $J$ and coherent sheaves $F_j$ and $G_j$ on $\mathcal{X}$ and integer $n_j$ for every $j\in J$. Moreover, if $K\in D_{\operatorname{coh}}^b(\mathcal{X}\times_R\mathcal{X})^{\leq 0}$, then so is $B_1$. (Here, $D_{\operatorname{coh}}^b(\mathcal{X}\times_R\mathcal{X})^{\leq 0}$ are those complexes $E$ with bounded coherent cohomology and that satisfy $\mathcal{H}^n (E) \cong 0$ for $n>0$).
\end{lemma}

\begin{proof}
    By \cite[Proposition 2.3]{Peng:2024}, there exists a finite set $J$, coherent sheaves $F_j$ and $G_j$ on $\mathcal{X}$ for each $j\in J$, and a morphism of complexes
    \begin{displaymath}
        \bigoplus_{j\in J}\bigg(g_1^\ast F_j\otimes g_2^\ast G_j[n_j]\bigg)\to K
    \end{displaymath}
    that induces surjective maps on cohomology sheaves. Set $A_0=\bigoplus_{j\in J}(g_1^\ast F_j\otimes g_2^\ast G_j[n_j])$ and $K_0=\bigoplus_{j\in J}(\mathbf{L}g_1^\ast F_j\otimes^\mathbf{L}\mathbf{L}g_2^\ast G_j[n_j])$. Note that $K_0\in D_{\operatorname{coh}}^b(\mathcal{X}\times_R\mathcal{X})$ since $\mathcal{X}\times_R\mathcal{X}$ is regular. Observe that there is a morphism $K_0\to K$ that induces surjective maps on cohomology sheaves. Indeed, we have $\mathcal{H}\!^{0}(\mathbf{L}g_1^\ast F_j\otimes^\mathbf{L}\mathbf{L}g_2^\ast G_j)\simeq g_1^\ast F_j\otimes g_2^\ast G_j$ since both projections $g_1,g_2$ are flat. Therefore, we get a canonical map$$(\mathbf{L}g_1^\ast F_j\otimes^\mathbf{L}\mathbf{L}g_2^\ast G_j)[n_j]\to \mathcal{H}\!^{0}(\mathbf{L}g_1^\ast F_j\otimes^\mathbf{L}\mathbf{L}g_2^\ast G_j)[n_j]\simeq g_1^\ast F_j\otimes g_2^\ast G_j[n_j]$$whose induced maps on cohomology sheaves are surjective. Taking the direct sum yields
    \begin{displaymath}
        K_0=\bigoplus_{j\in J}\bigg(\mathbf{L}g_1^\ast F_j\otimes^\mathbf{L}\mathbf{L}g_2^\ast G_j[n_j]\bigg)\longrightarrow A_0=\bigoplus_{j\in J}\bigg(g_1^\ast F_j\otimes g_2^\ast G_j[n_j]\bigg).
    \end{displaymath}
    Then the composition $K_0\to A_0\to K$ is a morphism with the desired property. We may then form the distinguished triangle
    \begin{displaymath}
        B_1\longrightarrow K_0\longrightarrow K\longrightarrow B_1[1],
    \end{displaymath}
    and the result follows. If $K\in D_{\operatorname{coh}}^b(\mathcal{X}\times_R\mathcal{X})^{\leq 0}$, we may remove the terms in $A_0$ and $K_0$ containing negative shifts. The resulting map $K_0\to A_0\to K$ will still induce surjective morphisms on cohomology sheaves. It follows that $B_1\in D_{\operatorname{coh}}^b(\mathcal{X}\times_R\mathcal{X})^{\leq 0}$ as desired. 
\end{proof}

\begin{corollary}
    \label{cor:derived_resolution}
    In the situation of \Cref{lem:derived_resolution_of_the_diagonal}, there exists a distinguished triangle in $D_{\operatorname{coh}}^b(\mathcal{X}\times_R\mathcal{X})$,
    \begin{displaymath}
        C_n\longrightarrow K\longrightarrow B_{n+1}[n+1]\longrightarrow C_n[1]
    \end{displaymath}
    for every $n\geq 0$ such that $C_n\in\langle\oplus_i^{n} K_i \rangle_{n+1}$ where
    \begin{displaymath}
        K_i=\bigoplus_{j\in J_{i}}\bigg(\mathbf{L}g_1^\ast F_i^j\otimes^{\mathbf{L}}\mathbf{L}g_2^\ast G_i^j[n_j]\bigg)
    \end{displaymath}
    for some finite sets $J_i$ and coherent sheaves $F_i^j$ and $G_i^j$ on $\mathcal{X}$ for every $j\in J_i$ and for every $0\leq i\leq n$. If $K\in D_{\operatorname{coh}}^b(\mathcal{X}\times_R\mathcal{X})^{\leq 0}$, it can be arranged that $B_i\in D_{\operatorname{coh}}^b(\mathcal{X}\times_R\mathcal{X})^{\leq 0}$ for every $0\leq i\leq n+1$.
\end{corollary}

\begin{proof}
    We proceed by induction $n$. If $n=0$, we have a distinguished triangle
    \begin{displaymath}
        B_1\longrightarrow K_0\longrightarrow K=B_0\longrightarrow B_1[1]
    \end{displaymath}
    by \Cref{lem:derived_resolution_of_the_diagonal}. Moreover, if $K=B_0\in D_{\operatorname{coh}}^b(\mathcal{X}\times_R\mathcal{X})^{\leq 0}$, then it can be arranged that $B_1$ is so. In this case, we have $C_0\simeq K_0$, and the statement is trivially true. Suppose there is a distinguished triangle
    \begin{displaymath}
        C_k\longrightarrow K\longrightarrow B_{k+1}[k+1]\longrightarrow C_k[1]
    \end{displaymath}
    for some $k$ such that $C_k\in\langle\oplus_i^{k} K_i \rangle_{k+1}$ where the $K_i$'s are of the desired form. By \Cref{lem:derived_resolution_of_the_diagonal}, we have a distinguished triangle
    \begin{displaymath}
        B_{k+2}\longrightarrow K_{k+1}\longrightarrow B_{k+1}\longrightarrow B_{k+2}[1]
    \end{displaymath}
    where $K_{k+1}$ is of the desired form. Moreover, if $B_{k+1}\in D_{\operatorname{coh}}^b(\mathcal{X}\times_R\mathcal{X})^{\leq 0}$, it can be arranged that $B_{k+2}$ is so. Let $C_{k+1}$ be the cocone of the composition
    \begin{displaymath}
        K\to B_1[1]\to B_2[2]\to\ldots\to B_{k+2}[k+2].
    \end{displaymath}
    By the octahedron axiom, there is a distinguished triangle
    \begin{displaymath}
        K_{k+1}[k]\longrightarrow C_{k}\longrightarrow C_{k+1}\longrightarrow K_{k+1}[k+1].
    \end{displaymath}
    Since $C_k\in\langle\oplus_i^{k+1} K_i\rangle_{k+1}$ and $K_{k+1}[k+1]\in\langle\oplus_i^{k+1} K_i\rangle_{1}$, we have $C_{k+1}\in\langle\oplus_i^{k+1} K_i\rangle_{k+2}$. By induction, the statement follows.
\end{proof}

\begin{proof}
    [Proof of \Cref{thm:upper_bound}]
    Note that $\mathcal{Y}$ has finite Krull dimension (see e.g.\ \cite[Lemma 5.1]{DeDeyn/Lank/ManaliRahul/Peng:2025}). Therefore, the desired upper bound is finite. Moreover, $\mathcal{Y}$ is a regular Noetherian algebraic stack because $f$ is of finite presentation with a regular Noetherian target. Set $D=\dim(f)+\operatorname{cd}(\mathcal{Y})$. Consider the fibred square
    \begin{displaymath}
        \begin{tikzcd}
            {\mathcal{Y}\times_{\operatorname{Spec}(R)} \mathcal{Y}} & {\mathcal{Y}} \\
            {\mathcal{Y}} & {\operatorname{Spec}(R).}
            \arrow["{g_2}", from=1-1, to=1-2]
            \arrow["{g_1}"', from=1-1, to=2-1]
            \arrow["f", from=1-2, to=2-2]
            \arrow["f"', from=2-1, to=2-2]
        \end{tikzcd}
    \end{displaymath}
    By base change, $g_1$ is smooth and finitely presented, and so $\mathcal{Y}\times_{\operatorname{Spec}(R)} \mathcal{Y}$ is a regular Noetherian algebraic stack. Hence, \cite[Proposition 3.6]{DeDeyn/Lank/ManaliRahul/Peng:2025} ensures that $\operatorname{Perf}(\#) = D^b_{\operatorname{coh}}(\#)$ for $\# =\mathcal{Y}$ and $\mathcal{Y}\times_{\operatorname{Spec}(R)} \mathcal{Y}$.

    Let $\Delta$ be the diagonal of $f$, which is finite by assumption. Therefore, we have $\Delta_\ast \mathcal{O}_{\mathcal{Y}}\cong \mathbf{R}\Delta_\ast \mathcal{O}_{\mathcal{Y}}\in \operatorname{Perf}(\mathcal{Y}\times_{\operatorname{Spec}(R)} \mathcal{Y})$. By \Cref{cor:derived_resolution}, we have the following distinguished triangle
    \begin{displaymath}
        C_D\longrightarrow \mathbf{R}\Delta_\ast \mathcal{O_Y}\longrightarrow B_{D+1}[D+1]\longrightarrow C_D[1]
    \end{displaymath}
    such that $C_D\in\langle\oplus_i^{D} K_i \rangle_{D+1}$ where each $K_i$ is of the form $\oplus_{j\in J_{i}}(\mathbf{L}g_1^\ast F_i^j\otimes^{\mathbf{L}}\mathbf{L}g_2^\ast G_i^j[n_i^j])$ for some finite sets $J_i$ and coherent sheaves $F_i^j$ and $G_i^j$ on $\mathcal{Y}$ for every $j\in J_i$ and for every $0\leq i \leq D$. Moreover, it can be arranged that $B_{D+1}\in D_{\operatorname{coh}}^b(\mathcal{Y}\times_R\mathcal{Y})^{\leq 0}$. It follows from \Cref{lem:exi^D_vanishing_regular_finite_Krull_dimension} that $\operatorname{Hom}(\mathbf{R}\Delta_\ast \mathcal{O_Y},B_{D+1}[D+1])=0$. This implies that $\mathbf{R}\Delta_\ast \mathcal{O_Y}$ is a direct summand of $C_D$ and thus contained in $\langle\oplus_i^{D} K_i\rangle_{D+1}$. Define $G_1=\oplus_{i,j}F_i^j$ and $G_2=\oplus_{i,j}G_i^j[n_i^j]$. Observe that $K_i$ is a direct summand of $\mathbf{L}g_1^\ast G_1\otimes^{\mathbf{L}}\mathbf{L}g_2^\ast G_2$ for every $0\leq i\leq D$. It follows that $\mathbf{R}\Delta_\ast \mathcal{O_Y}\in\langle\mathbf{L}g_1^\ast G_1\otimes^{\mathbf{L}}\mathbf{L}g_2^\ast G_2\rangle_{D+1}$ and thus $\dim_{\Delta} (f)\leq\dim (f)+\operatorname{cd}(\mathcal{Y})$.
    
    We now bound the Rouquier dimension of $D_{\operatorname{coh}}^b(\mathcal{Y})$. Since $\langle-\rangle$ is closed under shifts, we may suppress the shifts in our calculation below. For each $E\in \operatorname{Perf}(\mathcal{Y})$, flat base change and the projection formula imply
    \begin{displaymath}
        \begin{aligned}
            E
            &\cong \mathbf{R}(g_2)_\ast (\mathbf{L}g_1^\ast E \otimes^{\mathbf{L}} \mathbf{R}\Delta_\ast \mathcal{O}_{\mathcal{Y}}) 
            \\&\in \langle \mathbf{R}(g_2)_\ast(\mathbf{L}g_1^\ast E \otimes^{\mathbf{L}} (\bigoplus^D_{i=0} \bigoplus_{j\in J_{i}}(\mathbf{L}g_1^\ast F_i^j\otimes^{\mathbf{L}}\mathbf{L}g_2^\ast G_i^j))) \rangle_{D+1}
            \\&\subseteq \langle \bigoplus^D_{i=0} \bigoplus_{j\in J_{i}} \mathbf{R}(g_2)_\ast \mathbf{L}g_1^\ast (E\otimes^{\mathbf{L}} F^j_i ) \otimes^{\mathbf{L}} G^j_i \rangle_{D+1}
            \\&\subseteq \langle \bigoplus^D_{i=0} \bigoplus_{j\in J_{i}} \mathbf{L}f^\ast \mathbf{R}f_\ast (E\otimes^{\mathbf{L}} F^j_i ) \otimes^{\mathbf{L}} G^j_i \rangle_{D+1}.
        \end{aligned}
    \end{displaymath}
    Since $R$ is regular of finite Krull dimension, it follows that $\overline{\langle \mathcal{O}_{\operatorname{Spec}(R)} \rangle}_{\dim R +1}= D_{\operatorname{qc}}(\operatorname{Spec}(R))$ (see \Cref{ex:regular_ring}). Hence, 
    \begin{displaymath}
        \begin{aligned}
            E 
            &\in \langle \bigoplus^D_{i=0} \bigoplus_{j\in J_{i}} \mathbf{L}f^\ast \mathbf{R}f_\ast (E\otimes^{\mathbf{L}} F^j_i )  \otimes^{\mathbf{L}} G^j_i \rangle_{D+1}
            \\&\subseteq \overline{ \langle \bigoplus^D_{i=0} \bigoplus_{j\in J_{i}} \mathbf{L}f^\ast \mathcal{O}_{\operatorname{Spec}(R)} \otimes^{\mathbf{L}} G^j_i  \rangle}_{(D+1)(\dim R + 1)}
            \\&\subseteq \overline{\langle \bigoplus^D_{i=0} \bigoplus_{j\in J_{i}}  G^j_i \rangle}_{(D+1)(\dim R + 1)}.
        \end{aligned}
    \end{displaymath}
    By \Cref{lem:big_to_small_generation}, we obtain that $E\in \langle \oplus^D_{i=0} \oplus_{j\in J_{i}}  G^j_i \rangle_{(D+1)(\dim R + 1)}$. Since $\oplus^D_{i=0} \oplus_{j\in J_{i}}  G^j_i$ is independent of $E$, the result follows.
\end{proof}

\begin{proof}
    [Proof of \Cref{prop:ELS_fibre_product}]
    Consider the following diagram
    \begin{displaymath}
        \begin{tikzcd}
            {\mathcal{Y}^\prime_1 \times_{\mathcal{S}} \mathcal{Y}^\prime_2} & {\mathcal{Y}_1 \times_{\mathcal{S}} \mathcal{Y}^\prime_2} & {\mathcal{Y}^\prime_2} \\
            {\mathcal{Y}^\prime_1 \times_{\mathcal{S}} \mathcal{Y}_2} & {\mathcal{Y}_1 \times_{\mathcal{S}} \mathcal{Y}_2} & {\mathcal{Y}_2} \\
            {\mathcal{Y}^\prime_1} & {\mathcal{Y}_1} & {\mathcal{S}}
            \arrow["{h_1^{\prime \prime}}", from=1-1, to=1-2]
            \arrow["{h_2^{\prime \prime}}"', from=1-1, to=2-1]
            \arrow["{g^\prime_2}", from=1-2, to=1-3]
            \arrow["{h_2^\prime}", from=1-2, to=2-2]
            \arrow["{h_2}", from=1-3, to=2-3]
            \arrow["{h_1^\prime}", from=2-1, to=2-2]
            \arrow["{g_1^\prime}"', from=2-1, to=3-1]
            \arrow["{g_2}", from=2-2, to=2-3]
            \arrow["{g_1}"', from=2-2, to=3-2]
            \arrow["{f_2}", from=2-3, to=3-3]
            \arrow["{h_1}"', from=3-1, to=3-2]
            \arrow["{f_1}"', from=3-2, to=3-3]
        \end{tikzcd}
    \end{displaymath}
    whose faces are fibred squares. Since $\mathcal{S}$ is Noetherian, every stack in the diagram above is Noetherian. Before proceeding, we record several properties of the diagram, many of which follow by base change:
    \begin{itemize}
        \item By \cite[Lemma 2.3]{Hall/Lamarche/Lank/Peng:2025}, $h_1$ and $h_2$ are concentrated because each $f_i\circ h_i$ is such. So, each $h^\prime_i$, $h^{\prime \prime}_i$, $g_j \circ h^\prime_i$, and $g_i^\prime \circ h^{\prime \prime}_j$ are concentrated.  As each $f_i$ is flat and concentrated, base change implies each $g_j$ and $g^\prime_j$ is flat and concentrated.
        \item In the diagram above, each morphism is surjective (see e.g.\ \cite[\href{https://stacks.math.columbia.edu/tag/04XH}{Tag 04XH}]{StacksProject}), whereas all are flat except perhaps the $h_i, h_i^\prime, h_i^{\prime \prime}$. Furthermore, every algebraic stack above is Noetherian because each morphism is of finite presentation to a Noetherian algebraic stack.
        \item Each $h_i^\prime$ and $h_i^{\prime \prime}$ are proper because the $h_i$ are proper.
        \item Since $f_i\circ h_i$ is smooth, so is $g_j^\prime\circ h_i^{\prime \prime}$ when $i\not=j$, and so, $\mathcal{Y}^\prime_j \times_{\mathcal{S}} \mathcal{Y}^\prime_i$ is smooth over $\mathcal{S}$. 
        \item $\mathcal{Y}^\prime_1 \times_{\mathcal{S}} \mathcal{Y}^\prime_2$ and each $\mathcal{Y}^\prime_i$ are regular and Noetherian; so, \cite[Proposition 3.6]{DeDeyn/Lank/ManaliRahul/Peng:2025} implies $\operatorname{Perf}=D^b_{\operatorname{coh}}$ in such cases. 
        \item $\mathcal{Y}^\prime_1 \times_{\mathcal{S}} \mathcal{Y}^\prime_2$ and each $\mathcal{Y}^\prime_i$ have quasi-finite diagonal.
        To see, use the fact that each $f_i\circ h_i$ is quasi-DM, and so, their base changes along one another are too (i.e.\ $g^\prime_i \circ h^{\prime \prime}_j$ are quasi-DM). Hence, these algebraic stacks are quasi-DM over a quasi-DM algebraic stack, which ensures quasi-finite diagonal (see e.g.\ \cite[\href{https://stacks.math.columbia.edu/tag/050L}{Tag 050L}]{StacksProject}). 
    \end{itemize}

    Now, we prove the desired claim. In each case of $\mathcal{Y}^\prime_1 \times_{\mathcal{S}} \mathcal{Y}^\prime_2$ and each $\mathcal{Y}^\prime_i$, concentratedness ensures that $\operatorname{Perf}$ coincides with the compacts of $D_{\operatorname{qc}}$. Moreover, by \cite[Corollary 1.2]{Lank:2026}, $D_{\operatorname{qc}}$ is singly compactly generated for $\mathcal{Y}^\prime_1 \times_{\mathcal{S}} \mathcal{Y}^\prime_2$ and each $\mathcal{Y}^\prime_i$. So, $\operatorname{Perf}$ admits a classical generator for $\mathcal{Y}^\prime_1 \times_{\mathcal{S}} \mathcal{Y}^\prime_2$ and each $\mathcal{Y}^\prime_i$ (see e.g.\ \cite[\href{https://stacks.math.columbia.edu/tag/09SR}{Tag 09SR}]{StacksProject}). Let $P_i$ be classical generators for each $\operatorname{Perf}(\mathcal{Y}^\prime_i)$. From \cite[Corollary 5.10]{Neeman:2023} and \cite[Corollary 6.3]{Hall/Lamarche/Lank/Peng:2025}, it follows that 
    \begin{itemize}
        \item $\langle \mathbf{L}(g_1^\prime \circ h^{\prime \prime}_2)^\ast P_1 \otimes^{\mathbf{L}} \mathbf{L}(g_2^\prime \circ h^{\prime \prime}_1)^\ast P_2 \rangle = \operatorname{Perf}(\mathcal{Y}^\prime_1 \times_{\mathcal{S}} \mathcal{Y}^\prime_2)$
        \item $\langle \mathbf{R}(h^\prime_1 \circ h^{\prime \prime}_2)_\ast (\mathbf{L}(g_1^\prime \circ h^{\prime \prime}_2)^\ast P_1 \otimes^{\mathbf{L}} \mathbf{L}(g_2^\prime \circ h^{\prime \prime}_1)^\ast P_2) \rangle = D^b_{\operatorname{coh}}(\mathcal{Y}_1 \times_{\mathcal{S}} \mathcal{Y}_2)$
        \item $\langle \mathbf{R}(h_i)_\ast P_i \rangle = D^b_{\operatorname{coh}}(\mathcal{Y}_i)$ for each $i$.
    \end{itemize}
    Note that we use properness of the $h^\prime_i$ and $h^{\prime \prime}_i$ to apply \cite{Hall/Lamarche/Lank/Peng:2025}. Moreover, we use the separated diagonal assumption on $\mathcal{S}$ to apply \cite[Corollary 5.10]{Neeman:2023}. Indeed, the diagonal is quasi-finite and separated and thus quasi-affine by Zariski's main theorem.

    It suffices to show that
    \begin{displaymath}
        \langle \mathbf{L}g_1^\ast \mathbf{R}(h_1)_\ast P_1 \otimes^{\mathbf{L}} \mathbf{L}g_2^\ast \mathbf{R}(h_2)_\ast P_2  \rangle = D^b_{\operatorname{coh}}(\mathcal{Y}_1 \times_{\mathcal{S}} \mathcal{Y}_2).
    \end{displaymath}
    Indeed, given any other choice of classical generators $G_i$ for $D^b_{\operatorname{coh}}(\mathcal{Y}_i)$, we have $\langle \mathbf{L}g_i^\ast G_i \rangle = \langle \mathbf{L}g_i^\ast \mathbf{R}(h_i)_\ast P_i \rangle$ implies for $i\not=j$,
    \begin{displaymath}
        \langle \mathbf{L}g_i^\ast G_i \otimes^{\mathbf{L}} \mathbf{L}g_j^\ast G_j \rangle
        =\langle \mathbf{L}g_i^\ast G_i \otimes^{\mathbf{L}} \mathbf{L}g_j^\ast \mathbf{R}(h_j)_\ast P_j \rangle 
        = \langle \mathbf{L}g_i^\ast \mathbf{R}(h_i)_\ast P_i \otimes^{\mathbf{L}} \mathbf{L}g_j^\ast \mathbf{R}(h_j)_\ast P_j \rangle
    \end{displaymath}
    This can be seen using the fact $(-)\otimes^{\mathbf{L}} E$ is exact for $E\in \{ \mathbf{L}g_j^\ast \mathbf{R}(h_j)_\ast P_j, \mathbf{L}g_j^\ast G_j\}$.
    
    Towards proving the desired claim, there is a string of isomorphisms using flat base change and the projection formula
    \begin{displaymath}
        \begin{aligned}
            \mathbf{R} & (h^\prime_1 \circ h^{\prime \prime}_2)_\ast (\mathbf{L}(g_1^\prime \circ h^{\prime \prime}_2)^\ast P_1 \otimes^{\mathbf{L}} \mathbf{L}(g_2^\prime \circ h^{\prime \prime}_1)^\ast P_2) 
            \\&\cong \mathbf{R}(h^\prime_1)_\ast \mathbf{R}(h^{\prime \prime}_2)_\ast (\mathbf{L}(h^{\prime \prime}_2)^\ast \mathbf{L}(g_1^\prime)^\ast P_1 \otimes^{\mathbf{L}} \mathbf{L}(h^{\prime \prime}_1)^\ast \mathbf{L}(g_2^\prime)^\ast P_2)
            \\&\cong \mathbf{R}(h^\prime_1)_\ast (\mathbf{R}(h^{\prime \prime}_2)_\ast \mathbf{L}(h^{\prime \prime}_1)^\ast \mathbf{L}(g_2^\prime)^\ast P_2 \otimes^{\mathbf{L}} \mathbf{L}(g_1^\prime)^\ast P_1)
            \\&\cong \mathbf{R}(h^\prime_1)_\ast (\mathbf{L}(h^\prime_1)^\ast  \mathbf{L}g_2^\ast \mathbf{R}(h_2)_\ast P_2 \otimes^{\mathbf{L}} \mathbf{L}(g_1^\prime)^\ast P_1)
            \\&\cong \mathbf{L}g_1^\ast \mathbf{R}(h_1)_\ast P_1 \otimes^{\mathbf{L}} \mathbf{L}g_2^\ast \mathbf{R}(h_2)_\ast P_2.
        \end{aligned}
    \end{displaymath}
    However, we noted earlier that 
    \begin{displaymath}
        \langle \mathbf{R}(h^\prime_1 \circ h^{\prime \prime}_2)_\ast (\mathbf{L}(g_1^\prime \circ h^{\prime \prime}_2)^\ast P_1 \otimes^{\mathbf{L}} \mathbf{L}(g_2^\prime \circ h^{\prime \prime}_1)^\ast P_2) \rangle = D^b_{\operatorname{coh}}(\mathcal{Y}_1 \times_{\mathcal{S}} \mathcal{Y}_2).
    \end{displaymath}
    This completes the proof.
\end{proof}

\begin{proof}
    [Proof of \Cref{thm:compare_diagonal_morphism}]
    There is nothing to show if $\dim_{\Delta}(f_2)=\infty$. Thus, assume that $\dim_{\Delta}(f_2) < \infty$. Consider the commutative diagram of fibred squares,
    \begin{displaymath}
        \begin{tikzcd}
            {\mathcal{Y}_2 \times_{\mathcal{S}} \mathcal{Y}_2} & {\mathcal{Y}_1 \times_{\mathcal{S}} \mathcal{Y}_2} & {\mathcal{Y}_2} \\
            {\mathcal{Y}_2 \times_{\mathcal{S}} \mathcal{Y}_1} & {\mathcal{Y}_1 \times_{\mathcal{S}} \mathcal{Y}_1} & {\mathcal{Y}_1} \\
            {\mathcal{Y}_2} & {\mathcal{Y}_1} & {\mathcal{S}.}
            \arrow["{h_1^{\prime \prime}}", from=1-1, to=1-2]
            \arrow["{h_2^{\prime \prime}}"', from=1-1, to=2-1]
            \arrow["{g^\prime_2}", from=1-2, to=1-3]
            \arrow["{h_2^\prime}", from=1-2, to=2-2]
            \arrow["f", from=1-3, to=2-3]
            \arrow["{h_1^\prime}", from=2-1, to=2-2]
            \arrow["{g_1^\prime}"', from=2-1, to=3-1]
            \arrow["{g_2}", from=2-2, to=2-3]
            \arrow["{g_1}"', from=2-2, to=3-2]
            \arrow["{f_1}", from=2-3, to=3-3]
            \arrow["f"', from=3-1, to=3-2]
            \arrow["{f_1}"', from=3-2, to=3-3]
        \end{tikzcd}
    \end{displaymath}
    By base change and \cite[Lemma 2.3]{Hall/Lamarche/Lank/Peng:2025}, we know that each $h^\prime_i$ and $h^{\prime \prime}_i$ are concentrated because $f$ is concentrated. Also, base change implies that each $g_i$ and $g^\prime_i$ is concentrated. Denote by $\Delta_i\colon \mathcal{Y}_i \to \mathcal{Y}_i\times_\mathcal{S}\mathcal{Y}_i$ the diagonal morphism. From  \Cref{lem:alt_Delta_dim}, there is a $G\in D_{\operatorname{coh}}^b(\mathcal{Y}_2)$ such that 
    \begin{displaymath}
        \mathbf{R} (\Delta_2)_\ast \mathcal{O}_{\mathcal{Y}_2} \in \langle \mathbf{L} (g^\prime_1 \circ h^{\prime \prime}_2)^\ast G \otimes^{\mathbf{L}} \mathbf{L} (g^\prime_2 \circ h^{\prime \prime}_1)^\ast G \rangle_N
    \end{displaymath}
    where $N:= \dim_\Delta (f_2)+1$. Now, using a sequence of flat base change and projection formula \cite[Corollaries 4.12 \& 4.13]{Hall/Rydh:2017}, it follows that 
    \begin{displaymath}
        \begin{aligned}
            \mathbf{L}g_1^\ast \mathbf{R} f_\ast G \otimes^{\mathbf{L}} \mathbf{L}g_2^\ast \mathbf{R} f_\ast G
            &\cong \mathbf{R} (h^\prime_1)_\ast \mathbf{L} (g^\prime_1)^\ast G \otimes^{\mathbf{L}} \mathbf{L}g_2^\ast \mathbf{R} f_\ast G
            \\&\cong \mathbf{R} (h^\prime_1)_\ast (\mathbf{L} (g^\prime_1)^\ast G \otimes^{\mathbf{L}} \mathbf{L} (h^\prime_1)^\ast  \mathbf{L}g_2^\ast \mathbf{R} f_\ast G)
            \\&\cong \mathbf{R} (h^\prime_1)_\ast (\mathbf{L} (g^\prime_1)^\ast G \otimes^{\mathbf{L}} \mathbf{R} (h^{\prime \prime}_2)_\ast \mathbf{L} (h^{\prime \prime}_1)^\ast \mathbf{L} (g^\prime_2)^\ast G)
            \\&\cong \mathbf{R} (h^\prime_1)_\ast \mathbf{R} (h^{\prime \prime}_2)_\ast ( \mathbf{L} (h^{\prime \prime}_2)^\ast \mathbf{L} (g^\prime_1)^\ast G \otimes^{\mathbf{L}} \mathbf{L} (h^{\prime \prime}_1)^\ast \mathbf{L} (g^\prime_2)^\ast G)
        \end{aligned}
    \end{displaymath}
    However, the following diagram 
    \begin{displaymath}
        \begin{tikzcd}
            {\mathcal{Y}_2 \times_{\mathcal{S}} \mathcal{Y}_2} & {\mathcal{Y}_1 \times_{\mathcal{S}} \mathcal{Y}_1} \\
            {\mathcal{Y}_2} & {\mathcal{Y}_1}
            \arrow["{h^\prime_2 \circ h^{\prime \prime}_1}", from=1-1, to=1-2]
            \arrow["{\Delta_2}", from=2-1, to=1-1]
            \arrow["f"', from=2-1, to=2-2]
            \arrow["{\Delta_1}"', from=2-2, to=1-2]
        \end{tikzcd}
    \end{displaymath}
    commutes.
    Since $\mathcal{O}_{\mathcal{Y}_1}\xrightarrow{ntrl.} \mathbf{R} f_\ast \mathcal{O}_{Y_2}$ is an isomorphism, it follows that 
    \begin{displaymath}
        \begin{aligned}
            \mathbf{R} (\Delta_1)_\ast \mathcal{O}_{\mathcal{Y}_1} 
            &\cong \mathbf{R} (\Delta_1)_\ast \mathbf{R}f_\ast \mathcal{O}_{\mathcal{Y}_2}
            \\&\cong \mathbf{R} (h^\prime_2 \circ h^{\prime \prime}_1 \circ \Delta_{\mathcal{Y}_2})_\ast \mathcal{O}_{\mathcal{Y}_2}
            \\&\in \mathbf{R} (h^\prime_2 \circ h^{\prime \prime}_1)_\ast \langle \mathbf{L} (g^\prime_1 \circ h^{\prime \prime}_2)^\ast G \otimes^{\mathbf{L}} \mathbf{L} (g^\prime_2 \circ h^{\prime \prime}_1)^\ast G \rangle_N
            \\&\subseteq \langle \mathbf{R} (h^\prime_2 \circ h^{\prime \prime}_1)_\ast \mathbf{L} (g^\prime_1 \circ h^{\prime \prime}_2)^\ast G \otimes^{\mathbf{L}} \mathbf{L} (g^\prime_2 \circ h^{\prime \prime}_1)^\ast G \rangle_N
            \\&\subseteq \langle \mathbf{L}g_1^\ast \mathbf{R} f_\ast G \otimes^{\mathbf{L}} \mathbf{L}g_2^\ast \mathbf{R} f_\ast G \rangle_N.
        \end{aligned}
    \end{displaymath}
    Since $\mathbf{R}f_\ast$ preserves complexes with bounded and coherent cohomology, we see that $\mathbf{R}f_\ast G\in D_{\operatorname{coh}}^b(\mathcal{Y}_1)$. Thus, we obtain that $\dim_{\Delta}(f_1) \leq N-1 = \dim_{\Delta}(f_2)$. 
\end{proof}

\begin{reminder}
    \label{rmk:quotient_singularities}
    Let $k$ be a field. We follow \cite[\S 5]{Satriano:2012a}, \cite[\S 1.2]{Satriano:2012b}, and \cite[\S 6]{Bresciani/Vistoli:2024}. A finite type $k$-scheme $X$ has at worst \textbf{quotient singularities} if there is an \'{e}tale cover $\{ U_i/G_i \to X \}$ where $U_i$ is $k$-smooth and the $G_i$ are finite group schemes over $k$. We say $X$ has \textbf{tame quotient singularities} if each $G_i$ has order coprime to the characteristic of $k$; these are sometimes called \textit{good quotient singularities}. Lastly, we say $X$ has \textbf{linearly reductive singularities} if each $G_i$ is a finite linearly reductive group $k$-scheme. It is known that tame quotient singularities are linearly reductive singularities. 
\end{reminder}

\begin{example}
    \label{ex:char_zero_non_toric}
    Consider a finite subgroup $G\subseteq \operatorname{SL}(2,\mathbb{C})$ with action on $\mathbb{A}_{\mathbb{C}}^2$. Let $X=\mathbb{A}_{\mathbb{C}}^2/G$ be the quotient variety. Note that $X$ has quotient, and hence rational, singularities (see e.g.\ \cite[Corollary 7.4.10]{Ishii:2018}). Now, by \cite[Theorem 7.4.17]{Ishii:2018}, $X$ is toric if, and only if, $G$ is cyclic. Thus, for $G$ not cyclic, we have many examples of surfaces which have rational singularities and are not toric (e.g.\ see \cite{Nguyen/vanderPut/Top:2008}, \cite[Satz 2.9]{Brieskorn:1968}, \cite[p.\ 38]{Riemenschneider:1977}, and \cite[\S 10.1]{Coxeter:1991} for treatments of the classification of finite subgroups of $\operatorname{GL}(2,\mathbb{C})$; also, see \cite[Theorem 7.4.19]{Ishii:2018} or \cite[Satz 2.10]{Brieskorn:1968} for explicit nontoric cases).
\end{example}

\begin{proof}
    [Proof of \Cref{cor:compare_diagonal_morphism}]
    \hfill
    \begin{enumerate}
        \item Let $f\colon X \to \operatorname{Spec}(k)$ be a morphism from a separated integral scheme of finite type to a field of characteristic zero. Suppose $X$ has rational singularities. Let $g\colon \widetilde{X}\to X$ be a proper birational morphism from a smooth $k$-scheme (use e.g.\ \cite[Theorem 1.1]{Temkin:2008}). By \Cref{thm:upper_bound}, we know that$$\dim_{\Delta} (f\circ g)\leq \dim(f\circ g)+\operatorname{cd}(\widetilde{X})\leq 2\dim X.$$Since $g$ is proper and birational and $X$ has rational singularities, we see that $\mathcal{O}_X\cong\mathbf{R}g_\ast \mathcal{O}_{\widetilde{X}}$ and $\mathbf{R}g_\ast $ preserves bounded coherent complexes. Hence, it follows from \Cref{thm:compare_diagonal_morphism} that$$\dim_{\Delta}(f) \leq \dim_{\Delta}(f\circ g)\leq 2\dim X.$$ 
        \item Let $f\colon X\to \operatorname{Spec}(k)$ be an irreducible separated scheme of finite type over a perfect field $k$. Suppose that $X$ has at worst linearly reductive singularities. By \cite[Theorem 1.10]{Satriano:2012a}, there exists a $k$-smooth separated tame Artin stack $\mathcal{X}$ whose coarse moduli space is $\pi\colon \mathcal{X}\to X$. Since $\mathcal{X}$ is tame and separated, we see that $\operatorname{cd}(\mathcal{X})\leq\dim(X)$ by \Cref{ex:cohomological_dimension_good_moduli_space} and $\mathcal{X}$ has finite diagonal. By the Keel--Mori theorem \cite[Theorem 6.12]{Rydh:2013}, the map $\pi$ is a proper universal homeomorphism and $\mathcal{O}_X\cong\pi_\ast \mathcal{O}_{\mathcal{X}}$. It follows that $\dim(\mathcal{X})=\dim(X)$. By \Cref{thm:upper_bound}, we know that$$\dim_{\Delta} (f\circ\pi)\leq \dim(f\circ\pi)+\operatorname{cd}(\mathcal{X})\leq 2\dim(\mathcal{X})=2\dim X.$$Since $\mathcal{X}$ is tame, $\pi_\ast $ is exact. Since $\pi$ is proper, we see that $R\pi_\ast $ preserves bounded coherent complexes. Hence, it follows from \Cref{thm:compare_diagonal_morphism} that$$\dim_{\Delta}(f) \leq \dim_{\Delta}(f\circ \pi)\leq 2\dim X.$$
    \end{enumerate}
\end{proof}

\bibliographystyle{alpha}
\bibliography{mainbib}

\end{document}